\documentclass[a4paper,12pt]{article}


\usepackage{amsthm}

\usepackage{amsmath,amssymb}
\usepackage{mathtools}

\usepackage[margin=25mm]{geometry}
\usepackage{abstract}
\usepackage[T1]{fontenc}
\usepackage[utf8]{inputenc}

\usepackage[expansion=false,nopatch=footnote]{microtype}
\usepackage[lining,tabular,scale=.9]{FiraSans}          
\usepackage[tt=false,semibold]{libertinus}              
\usepackage[libertine,smallerops]{newtxmath}            
\usepackage[lining,scaled=.8]{FiraMono}                 
\usepackage[cal=boondox,frak=euler]{mathalpha}          

\DeclareFontEncoding{MDA}{}{}
\DeclareSymbolFont{mathdesignA}{MDA}{mdput}{m}{n}
\SetSymbolFont{mathdesignA}{bold}{MDA}{mdput}{b}{n}
\DeclareSymbolFontAlphabet{\mathbb}{mathdesignA}

\DeclareFontFamily{OMX}{MnSymbolE}{}
\DeclareSymbolFont{MnLargeSymbols}{OMX}{MnSymbolE}{m}{n}
\DeclareFontShape{OMX}{MnSymbolE}{m}{n}{%
  <-6> MnSymbolE5 <6-7> MnSymbolE6 <7-8> MnSymbolE7 <8-9> MnSymbolE8
  <9-10> MnSymbolE9 <10-12> MnSymbolE10 <12-> MnSymbolE12}{}
\DeclareMathDelimiter{[}{\mathopen}{MnLargeSymbols}{'000}{MnLargeSymbols}{'000}
\DeclareMathDelimiter{]}{\mathclose}{MnLargeSymbols}{'005}{MnLargeSymbols}{'005}
\DeclareMathDelimiter{\llbr}{\mathopen}{MnLargeSymbols}{'102}{MnLargeSymbols}{'102}
\DeclareMathDelimiter{\rrbr}{\mathclose}{MnLargeSymbols}{'107}{MnLargeSymbols}{'107}

\usepackage{setspace}
\newcommand{\initlengths}{%
  \setlength{\abovedisplayshortskip}{3pt plus 9pt minus 3pt}%
  \setlength{\belowdisplayshortskip}{9pt plus 9pt minus 9pt}%
  \setlength{\abovedisplayskip}{9pt plus 9pt minus 9pt}%
  \setlength{\belowdisplayskip}{9pt plus 9pt minus 9pt}%
  \hfuzz 1pt%
  \tolerance 400%
}

\usepackage{titlesec}
\usepackage{titling}

\pretitle{\vspace{0pt}\setstretch{1.05}\selectfont\LARGE\centering}
\posttitle{\par\vspace{6pt}}

\titleformat{\section}{\Large}{\thesection}{1em}{}
\titleformat{\subsection}{\large\sffamily\bfseries\boldmath}{\thesubsection}{1em}{}

\usepackage{enumitem}
\setlist{noitemsep}
\setlist[enumerate]{label=\textnormal{(\roman*)}}

\usepackage[bottom,hang,multiple]{footmisc}
\usepackage{etoolbox}
\makeatletter
\patchcmd{\@makefntext}{\ifFN@hangfoot\bgroup}%
{\ifFN@hangfoot\bgroup\def\@makefnmark{\rlap{\sffamily\bfseries\@thefnmark}}}{}{}%
\makeatother

\usepackage[labelsep=period,labelfont={bf}]{caption}

\usepackage{xcolor}
\definecolor{paperblue}{RGB}{0,102,204}
\usepackage[
  colorlinks=true,
  linkcolor=paperblue,
  citecolor=paperblue,
  urlcolor=paperblue
]{hyperref}

\usepackage[
  backend=biber,
  style=alphabetic,
  maxnames=10,
  giveninits=true,
  sortcites=true,      
  sorting=nyt          
]{biblatex}
\newcommand{\preparebibliography}{
  \phantomsection
  \addcontentsline{toc}{section}{References}
  \sloppy
  \setstretch{1.1}
  \renewcommand*{\bibfont}{\normalfont\small}
}

\usepackage[capitalize,nameinlink,noabbrev]{cleveref}

\newcommand{\C}{\mathbb{C}}

\DeclareMathOperator{\supp}{supp}

\renewcommand{\d}{\mathop{}\!\mathrm{d}}

\newcommand{\cU}{\mathcal{U}}

\newcommand{\lp}{\langle}
\newcommand{\rp}{\rangle}

\usepackage{aliascnt}

\theoremstyle{plain}
\newtheorem{theorem}{Theorem}[section]

\newaliascnt{lemma}{theorem}
\newtheorem{lemma}[lemma]{Lemma}
\aliascntresetthe{lemma}

\newaliascnt{proposition}{theorem}
\newtheorem{proposition}[proposition]{Proposition}
\aliascntresetthe{proposition}

\newaliascnt{corollary}{theorem}

\aliascntresetthe{corollary}

\theoremstyle{definition}

\newaliascnt{definition}{theorem}
\newtheorem{definition}[definition]{Definition}
\aliascntresetthe{definition}

\newaliascnt{example}{theorem}

\aliascntresetthe{example}

\theoremstyle{remark}

\newaliascnt{remark}{theorem}
\newtheorem{remark}[remark]{Remark}
\aliascntresetthe{remark}

\crefname{theorem}{theorem}{theorems}
\Crefname{theorem}{Theorem}{Theorems}

\crefname{lemma}{lemma}{lemmas}
\Crefname{lemma}{Lemma}{Lemmas}

\crefname{proposition}{proposition}{propositions}
\Crefname{proposition}{Proposition}{Propositions}

\crefname{corollary}{corollary}{corollaries}
\Crefname{corollary}{Corollary}{Corollaries}

\crefname{definition}{definition}{definitions}
\Crefname{definition}{Definition}{Definitions}

\crefname{example}{example}{examples}
\Crefname{example}{Example}{Examples}

\crefname{remark}{remark}{remarks}
\Crefname{remark}{Remark}{Remarks}

\newenvironment{statement}[1]
  {\par\medskip\noindent\textbf{#1.}\itshape\ }
  {\par\medskip}

\title{Equidistribution of graphs of holomorphic correspondences}
\author{Muhan Luo\thanks{
Department of Mathematics, National University of Singapore, Singapore 119076. 
Email: \texttt{e708207@u.nus.edu}
}
}

\date{}

\begin{document}
\initlengths
\maketitle

\begin{abstract}
\setstretch{1.1}

Let $X$ be a compact Riemann surface. Let $f$ be a holomorphic self-correspondence of $X$ with degrees $d_1$ and $d_2$. Assume that $d_1\neq d_2$ or $f$ is non-weakly modular. We show that the graphs of the iterates $f^n$ of $f$ are equidistributed exponentially fast with respect to a positive closed current in $X\times X$.

\end{abstract}

{\setstretch{1.1}\tableofcontents}

\section{Introduction and main results}

Let $X$ be a compact Riemann surface. Let $\pi_1$ and $\pi_2$ be the canonical projections from $X\times X$ to its factors. A \textit{holomorphic correspondence} on $X$ is an effective analytic cycle $\Gamma = \sum_i \Gamma_i$ in $X\times X$ of pure dimension one containing no fiber of $\pi_1$ or $\pi_2$ where all $\Gamma_i$ are irreducible but not necessarily distinct. $\Gamma$ determines a multi-valued map $f$ on $X$: for any $x\in X$, define
\[
    f(x):=\pi_2(\pi_1^{-1}(x)\cap\Gamma),
\]
where the points are counted with multiplicity. We call $\Gamma$ the \textit {graph} of $f$. We define the \textit {degrees} of $f$ to be the degrees of $\pi_1|_\Gamma$ and $\pi_2|_\Gamma$ and denote them by $d_1(f)$ and $d_2(f)$ respectively. Then $f(x)$ is a set of $d_1(f)$ points counted with multiplicity. The adjoint of $f$ is defined by exchanging $\pi_1$ and $\pi_2$:
\[
    f^{-1}(x):=\pi_1(\pi_2^{-1}(x)\cap\Gamma).
\]

We can compose two correspondences and in particular consider the $n^{th}$ iterates $f^n$ of a holomorphic correspondence $f$ (see below for precise definition). Denote by $d_1=d_1(f)$ and $d_2=d_2(f)$.  Equidistribution of periodic points with respect to an invariant measure is one of the key questions in studying the dynamics of a holomorphic correspondence. Let $\Gamma_n$ denote the graph of $f^n$ in $X\times X$ which defines a positive closed $(1,1)$-current $[\Gamma_n]$. Periodic points of period $n$ can be identified with the intersection of $\Gamma_n$ with the diagonal of $X\times X$. It is therefore important to know the limit of $d_2^{-n}[\Gamma_n]$ when $n$ goes to infinity as well as the rate of convergence. In this paper, we prove that for two large classes of holomorphic correspondences, the normalized graph currents converge exponentially fast to a limit current which is related to the equilibrium measures. 

The dynamics of $f$ depends on whether $d_1$ equals to $d_2.$ Consider first the case when $d_1\neq d_2$. We may assume $d_1<d_2$ since the opposite case can be treated in the same way. A class of polynomial correspondences satisfying this condition is studied in \cite{D05BSMF}. The general case is studied by Dinh-Sibony \cite{DS06CMH}. See also \cite{fornaess:survey,S99PS} for the case of holomorphic maps. By \cite{DS06CMH}, the correspondence $f$ admits an equilibrium probability measure $ \mu $ such that
$f^*(\mu)=d_2\mu$. Equidistribution of periodic points with respect to $\mu$ can be obtained from \cite{DNT15}. Although the original proof is for meromorphic self-maps with dominant dynamical degree, it also works in this case, see also the survey \cite{DS17}. We prove the following result for $\Gamma_\infty:=\pi_1^*(\mu).$

\begin{theorem}\label{thm:main-thm-neq}
    Let $f$ be a holomorphic correspondence on a compact Riemann surface $X$ with degrees $d_1<d_2$. Let $\mu,\Gamma_n$ and $\Gamma_\infty$ be as above. Then for every $\alpha>0$, there is a constant $0<\lambda_\alpha<1$ such that for any test $(1,1)$-form $\beta$ of class $\mathcal C^\alpha$ on $X\times X$, we have 
    \begin{equation}\label{eq:main-thm-poly}
        \big| \langle d_2^{-n}[\Gamma_n]-\Gamma_\infty, \beta  \rangle \big|\leq C_\alpha\|\beta\|_{\mathcal C^\alpha}\lambda_\alpha^n, \quad \text{ for every }n\geq 1,
    \end{equation}
     where $C_\alpha>0$ is a constant independent of $n$ and $\beta$.
\end{theorem}

In the case $d_1=d_2=d$, less is known about the dynamical behaviours while some results are obtained for the subclass of modular correspondences, see \cite{clozel-ullmo,mok:correspondences}. In order to study more general cases, Dinh-Kaufmann-Wu \cite{DKW20} introduced the notion of non-weakly modular correspondences and they constructed two probability measures $\mu^+$ and $\mu^-$ on $X$ which are invariant in the sense that
\[
    f^*(\mu^+)=d\mu^+ \quad \text{and} \quad f_*(\mu^-)=d\mu^-.
\]
Exponential mixing properties and equidistribution of the images and pre-images are obtained with respect to $\mu^\pm$. 

For the distribution of periodic points when $d_1=d_2$, some results are obtained in \cite{CO01,dinh:modular} for modular correspondences. As far as we know, the exponential rate of convergence for various equidistribution problems is still open in this setting. Recently, Matus de la Parra \cite{Matus24} proved equidistribution of periodic points for a class of weakly modular but non-modular correspondences. However, for non-weakly modular correspondences the problem is still open. It is proven in \cite{DKW20} that in this case $d^{-n}[\Gamma_n]$ converges to $\Gamma_\infty:=\pi_1^*(\mu^+)+\pi_2^*(\mu^-)$ but without knowing the speed of convergence. We prove the following more precise result which is analogous to \Cref{thm:main-thm-neq}. It suggests that about half of the periodic points are repelling and equidistributed with respect to $\mu^+$ and half of them are attractive and equidistributed with respect to $\mu^-$. 

 \begin{theorem}\label{thm:main-thm-nwmod}
    Let $f$ be a non-weakly modular correspondence on a compact Riemann surface $X$ with degrees $d_1=d_2=d$. Let $\mu^+,\mu^-,\Gamma_n $ and $\Gamma_\infty$ be as above. Then for every $\alpha>0$, there is a constant $0<\lambda_\alpha<1$ such that for any test $(1,1)$-form $\beta$ of class $\mathcal C^\alpha$ on $X\times X$, we have 
    \begin{equation}\label{eq:main-thm-nwmod}
        \big| \langle d^{-n}[\Gamma_n]-\Gamma_\infty, \beta  \rangle \big|\leq C_\alpha\|\beta\|_{\mathcal C^\alpha}\lambda_\alpha^n, \quad \text{ for every }n\geq 1,
    \end{equation}
     where $C_\alpha>0$ is a constant independent of $n$ and $\beta$.
\end{theorem}

\begin{remark}\rm
    By the proof of \Cref{lemma:interpolation}, for both \Cref{thm:main-thm-neq} and \Cref{thm:main-thm-nwmod}, we can choose $\lambda_\alpha=\lambda_5$ when $\alpha\geq 5$ and $\lambda_\alpha=\lambda_5^{\alpha/5}$ for $0<\alpha<5.$ For \Cref{thm:main-thm-neq}, it is also clear from our proof that the constant $\lambda_5$ can be chosen to be any constant strictly larger than $\delta:=d_1d_2^{-1}$. Therefore, in this setting the constants $\lambda_\alpha$'s are independent of $f$ of given degrees $d_1$ and $d_2$. 
\end{remark}

Here is the main idea of our proof. First we notice that by interpolation theory, it suffices to prove \Cref{thm:main-thm-neq} or \Cref{thm:main-thm-nwmod} for forms of class $\mathcal C^5$, i.e. for $\alpha=5$ (see \cite{T78Book} and \Cref{lemma:interpolation} below). After choosing a good atlas, we work in an open chart $U\times U'$ on $X^2$ with complex coordinate $(x,y)$, where $U$ and $U'$ are charts on $X$. In this local setting, we reduce the problem to three cases:
\begin{enumerate}
    \item $\beta = \phi(x,y) \d x\wedge \d \Bar{x}$;
    \item $\beta = \phi(x,y) \d y\wedge \d \Bar{y}$;
    \item $\beta = \phi(x,y) \d x\wedge \d \Bar{y} \quad \text{or} \quad \beta = \phi(x,y) \d y\wedge \d \Bar{x}$.
\end{enumerate}

To prove \Cref{thm:main-thm-neq}, Case (1) can be done by direct computation using $d_1<d_2$. An application of Cauchy-Schwartz inequality then implies Case (3). To deal with Case (2), we use Fourier expansion to write $\beta$ as linear combinations of $\pi_1^*\varphi_I\wedge \pi_2^*\theta_I$ with controllable error. Here $\varphi_I$ and $\theta_I$ are smooth $(0,0)$ and $(1,1)$-forms on $X$ respectively. To complete the proof, we apply some equidistribution property of functions (\Cref{prop:equi-function-neq}) to $\varphi_I$. 

The proof for \Cref{thm:main-thm-nwmod} is analogous. We treat Case (1) and Case (2) in the same way as Case (2) of \Cref{thm:main-thm-neq}. The key point is still use of an equidistribution property (\Cref{prop:equi-function-nwmod}) which is parallel to \Cref{prop:equi-function-neq}.  For Case (3), we use Fourier expansion to reduce the test form to $\pi_1^*\gamma_I\wedge \pi_2^*\omega_I$ where $\gamma_I$ and $\omega_I$ are $(1,0)$ and $(0,1)$-forms. For such forms, an application of Cauchy-Schwartz inequality and contraction of the operator $d^{-1}f_*$ over $L^2_{(1,0)}$ (\Cref{normless1}) finishes the proof.

Finally we note that the main results and their proofs still hold when we use cycles with positive real coefficients to define correspondences which is useful in the study of random dynamics, see for example \cite{DKW21,DKW23}.

\subsection*{Acknowledgment} The author would like to thank the anonymous referee for valuable comments and suggestions, which have improved the presentation of this article. 

\section{Preliminary results}

\subsection*{Background on correspondences} We briefly recall some basic notions related to holomorphic correspondences. Let $f$ and $g$ be two correspondences on $X$ with graphs $\Gamma$ and $\Gamma'$ respectively. We consider the product $\Gamma \times \Gamma'$ in $X^4 = \{(x_1,x_2,x_3,x_4) : x_i \in X\}$. Define $\widehat{\Gamma}_{f\circ g}$ as the intersection $(\Gamma\times \Gamma')\cap\{x_2=x_3\}$. Let $\Pi_{1,4}$ be the canonical projection from $X^4$ to $X\times X$ which maps $(x_1,x_2,x_3,x_4)$ to $(x_1,x_4)$. Then the graph of the composition $f\circ g$ is given by the cycle
\[
    [\Gamma_{f\circ g}]:=(\Pi_{1,4})_*[\widehat{\Gamma}_{f\circ g}].
\]
The push-forward here could be understood in the sense of currents which is well-defined since we work on a compact manifold. For basic properties of currents, we refer the readers to \cite{D12Book}. Alternatively, we can define the $f\circ g$ as a multi-valued function whose values are given by
\[
    f\circ g(x)=\{z\in X: \exists\ y\in g(x)\text{ such that }z\in f(y)\}.
\]
The points are counted with multiplicity. We can see from the definition that composition of holomorphic correspondences is associative. Moreover, the degrees satisfy the simple relation $d_i(f \circ g) = d_i(f) \cdot d_i(g)$ for $i=1,2$. This allows us to consider iterates $f^n$ of order $n$ of $f$ and we have $d_i(f^n) = d_i(f)^n$ for every $n\geq 1$. 

A correspondence induces push-forward and pullback operators on currents. When $S$ is a smooth form, a continuous function or a finite measure, we have
\begin{equation}\label{eq:pullback}
    f_*(S):=(\pi_2)_*(\pi_1^*(S)\wedge [\Gamma]) \quad \text{and} \quad f^*(S):=(\pi_1)_*(\pi_2^*(S)\wedge[\Gamma]).
\end{equation}
When $S$ is a smooth form, $f^*(S)$ and $f_*(S)$ are smooth outside some finite sets. When $\varphi$ is a continuous function we have $f_*\varphi(y) = \sum_{x \in f^{-1}(y)} \varphi(x)$ where the points in $f^{-1}(y)$ are counted with multiplicity. This function is continuous. Therefore by duality, if $\delta_y$ is the Dirac measure at $y$, we have $f^*(\delta_y)=\sum_{x\in f^{-1}(y)}\delta_x.$ In general, if $\nu$ is a probability measure on $X$, then $f_*\nu$ and $f^*\nu$ are positive measures on $X$ of mass $d_1(f)$ and $d_2(f)$ respectively.

\subsection*{Action on \texorpdfstring{$L^2_{(1,0)}$}{L2(1,0)} and equidistribution properties} 
 Let $L^2_{(1,0)}$ be the space of (1,0)-forms on $X$ with $L^2$ coefficients. For each $\gamma\in L^2_{(1,0)}$, its $L^2$ norm is given by
\[
    \|\gamma\|_{L^2}=\left(\int_X \sqrt{-1}\gamma\wedge\Bar{\gamma}\right)^{1/2}.
\]
Let $f$ be a holomorphic correspondence on $X$ and its degrees are denoted by $d_1$ and $d_2$. The action of $f^*$ has been defined on smooth (1,0)-forms in (\ref{eq:pullback}). By extending continuously, we can also define the pullback $f^*$ on $L^2_{(1,0)}$. When $d_1=d_2=d$, by \cite[Proposition 2.1]{DKW20}, the norm of $f^*$ is bounded by $d$. But in general, $d^{-1}f^*$ is not necessarily a contraction. Non-weakly modular correspondences are defined for this case. Their pullback actions on $L^2_{(1,0)}$ are also contracting.

\begin{definition}[\cite{DKW20}, Definition 3.1]\rm
    A correspondence $f$ on $X$ with degrees $d_1=d_2=d$ is called \textit{ non-weakly modular of degree} $d$ if there does not exist a positive measure $m$ on its graph $\Gamma$ and probability measures $m_1$ and $m_2$ on $X$ such that $m=(\pi_1|_\Gamma)^*(m_1)$ and $m=(\pi_2|_\Gamma)^*(m_2).$
\end{definition}

The results are summarized in the following:

\begin{proposition}[\cite{DKW20}, Proposition 3.1]\label{normless1}
    Let $f$ be a non-weakly modular holomorphic correspondence of degree $d$ on a compact Riemann surface $X$. Consider the operators $d^{-1}f^*$ and $d^{-1}f_*$ acting on $L^2_{(1,0)}$. Then there is a constant $0<\lambda<1$ such that $\|d^{-1}f^*\|<\lambda$ and $\|d^{-1}f_*\|<\lambda$.
\end{proposition}

The above proposition allows the authors in \cite{DKW20} to construct the canonical invariant measures which are mentioned in the introduction. In particular, they obtain an equidistribution property for certain class of functions with respect to these measures. 

\begin{proposition}[\cite{DKW20}, Proposition 3.2]\label{prop:equi-function-nwmod}
    Let $f$ be a non-weakly modular correspondence of degree $d$ on a compact Riemann surface $X$. Let $\mu^+ $ and $ \mu^-$ be as in \Cref{thm:main-thm-nwmod} and $\lambda$ be as in \Cref{normless1}. Then for every $\mathcal C^1$ function $\psi$ and every $n\geq 1$, we have
    \[
       \|d^{-n}(f^n)_*\psi-\lp \mu^+,\psi\rp\|_{L^1}\leq A\lambda^n\|\psi\|_{\mathcal C^1}
    \]
    where $A>0$ is independent of $n$ and $\psi.$ The same holds for $f^*$ and $\mu^-.$
\end{proposition}

The following analogous result is obtained in the proof of \cite[Theorem 5.1]{DS06CMH} when $f$ has distinct degrees.

\begin{proposition}\label{prop:equi-function-neq}
    Let $f$ be a holomorphic correspondence on a compact Riemann surface $X$ with degrees $d_1<d_2$. Let $\mu$ be as in \Cref{thm:main-thm-neq}. Then for every $\mathcal C^2$ function $\varphi$ and every $n\geq 1$, we have
    \[
       \|d_2^{-n}(f^n)_*\varphi-\lp \mu,\varphi\rp\|_{L^1}\leq A_0\delta^n\|\varphi\|_{\mathcal C^2}
    \]
    where $\delta=d_1d_2^{-1}<1$ and $A_0>0$ is a constant independent of $n$ and $\varphi.$
\end{proposition}

It should be noted here that the inequalities in \Cref{prop:equi-function-nwmod} and \Cref{prop:equi-function-neq} are slightly weaker than the original versions in the citations where $W^{1,2}$-norm and DSH-norms are involved, respectively.

\subsection*{Fourier expansion of periodic functions on \texorpdfstring{$\mathbb R^4$}{R4}} 

In order to apply the previous results to prove our main theorem, we use Fourier expansion to separate the variables. We review some Fourier analysis on $\mathbb R^4$ that will be used for $X^2.$ All the proofs of the results can be found in classical textbooks, for example \cite{SS}. Let $(x,y)=(x_1,x_2,y_1,y_2)\in \mathbb R^4$ and $\phi=\phi(x,y)$ be a function of class $\mathcal C^k$ on $\mathbb R^4$ with $k\geq 1$ which is periodic of period 1 in each variable. For any $I=(i_1,i_2,i_3,i_4)\in\mathbb Z^4$, define $I\cdot (x,y)=i_1x_1+i_2x_2+i_3y_1+i_4y_2$. Then by classical Fourier analysis, we have 
\begin{equation}\label{eq:fouri-expans}
    \phi(x,y)=\sum_{I\in \mathbb Z^4} a_Ie^{2\pi  \sqrt{-1}I\cdot (x,y)}.
\end{equation}
Here the convergence should be understood as pointwise convergence which is also uniform in our setting. The constants $a_I$ are given by
\begin{equation*}
    a_I=\int_{[0,1]^4}{\phi}(x,y)e^{-2\pi  \sqrt{-1} I\cdot (x,y)}\d x\d y.
\end{equation*}
For simplicity, we may assume $\|\phi\|_{\mathcal C^k}\leq 1.$ Then a priori $|a_I|\leq 1$ for all $I$. Define $|I|:=\max_{1\leq s\leq 4}\{|i_s|\}$. The speed of decay of the coefficients $a_I$ as $|I|$ tends to infinity is connected with the regularity of $\phi$. When $k=1$ and suppose $|I|=|i_1|>0$, using integration by parts we have 
\[
    |a_I|=\left|\int_{[0,1]^4}\frac{\partial \phi}{\partial x_1}\frac{e^{2\pi 
  \sqrt{-1} I\cdot (x,y)}}{2\pi \sqrt{-1}i_1}\d x\d y \right|\leq \frac{1}{2\pi  |i_1|}<\frac{1}{|I|}.
\]
By induction, we can prove that when $\phi$ is $\mathcal C^k$ for some $k\in\mathbb N$, for any $I\neq 0$ we have
\begin{equation}\label{eq:decay-fouri-coeffi}
    |a_I|\leq \frac{1}{|I|^k}.
\end{equation}

\section{Proof of the main theorems}

Let $f$ be a holomorphic correspondence on a compact Riemann surface $X$ and its degrees are denoted by $d_1$ and $d_2$ as in the introduction. We are in one of the two cases:
\begin{enumerate}[label=(\Roman*)]
    \item $d_1<d_2$ and $\Gamma_\infty=\pi^*_1(\mu)$;
    \item $d_1=d_2=d$, $f$ is non-weakly modular and $\Gamma_\infty=\pi_1^*(\mu^+)+\pi_2^*(\mu^-)$. 
\end{enumerate}
As we have seen, these two cases share some similarities in their dynamical behaviours. Therefore we will prove the main theorems in a unified way.

\subsection*{Preliminary settings}  Fix a K\"ahler form $\omega$ of $X$ with $\int_X\omega=1$. Then $\Omega=\frac{1}{\sqrt{2}}(\pi_1^*\omega+\pi_2^*\omega)$ is a K\"ahler form of $X\times X$ with $\int_{X\times X}\Omega^2=1$. Let $\Gamma_n$ be the graph of $f^n$ on $X\times X$. Note that for all $n\geq 1$ and $i=1$ or $2$, $\pi_i|_{\Gamma_n}$ is a ramified covering of $X$ and the ramification points are finite. In particular, they are of Lebesgue measure zero. By restricting to some connected and simply connected open subsets outside the ramification values, we can check that 
$(\pi_i|_{\Gamma_n})_*\beta$ is an $L^1$-form on X and we have the following property: for any smooth (1,1)-form $\beta$ on $X\times X$,
\begin{equation}\label{eq:graph-covering}
    \lp [\Gamma_n], \beta\rp=\lp [X],(\pi_i|_{\Gamma_n})_*\beta\rp= \int_X(\pi_i|_{\Gamma_n})_*\beta, \quad i=1,2.
\end{equation}
Recall that the mass of a positive closed $(1,1)$-current $T$ on $X\times X$ with respect to $\Omega$ is given by $\|T\|:=\lp T,\Omega\rp.$ Therefore,
\[
    \|[\Gamma_n]\|=\frac{1}{\sqrt{2}}\int_X(f^n)_*\omega+(f^n)^*\omega=\frac{d_1^n+d_2^n}{\sqrt{2}}.
\]
Since in both cases $d_1\leq d_2$, we have $\|d_2^{-n}[\Gamma_n]\|<2$ for all $n\geq 1$. On the other hand, for any probability measure $\nu$ on $X$, it is easy to compute that
\[
    \lp \pi_1^*(\nu),\Omega\rp=\lp \pi_2^*(\nu),\Omega\rp=\frac{1}{\sqrt{2}}.
\]
Therefore in either case $\|\Gamma_\infty\|<2$. Define
\begin{equation}\label{eq:T_n}
    T_n:=d_2^{-n}[\Gamma_n]-\Gamma_\infty.
\end{equation}
In both cases we have $\|T_n\|\leq 4$.

\begin{lemma}\label{lemma:interpolation}
    Assume that \Cref{thm:main-thm-neq} and \Cref{thm:main-thm-nwmod} hold for $\alpha=5$. Then they hold for all $\alpha>0.$
\end{lemma}
\begin{proof}
    When $\alpha>5$, the result directly follows from the assumption with $\lambda_\alpha=\lambda_5$. Suppose $0<\alpha<5$. For each $l\geq 0$, let $\mathcal E^l$ be the space of $(1,1)$-forms of class $\mathcal C^l$ on $X$ with the usual $\mathcal C^l$-norm. Fix $n\geq 1$. For each $l\geq 0$, let $\|T_n\|_{\mathcal C^{-l}}$ be the norm of $T_n$ as a continuous linear functional on $\mathcal E^l$, i.e.,
    \[
        \|T_n\|_{\mathcal C^{-l}}:=\sup_{\substack{\beta\in\mathcal E^l,\\ \|\beta\|_{\mathcal C^l}\leq 1}}\lp T_n,\beta\rp.
    \]
    In either case, our assumption implies
\[
    \|T_n\|_{\mathcal C^{-5}}\leq C_5\lambda_5^n
\]
 for some constants $0<\lambda_5<1$ and $C_5>0.$ On the other hand, we also have $\|T_n\|_{\mathcal C^0}\leq 4.$  By interpolation theory (see \cite{T78Book}), we obtain for each $0<\alpha<5$,
\[
    \|T_n\|_{\mathcal C^{-\alpha}}\leq C_{\alpha}\lambda_{\alpha}^n
\]
for some $C_\alpha>0$ and $\lambda_\alpha=\lambda_5^{\alpha/5}$. This directly implies \Cref{thm:main-thm-neq} and \Cref{thm:main-thm-nwmod} for all $\alpha>0.$
\end{proof}

From now on we take $\alpha=5$. Denote by $\mathbb U=(0,1)\times (0,1)$ and $\mathbb U_0=(\frac1{4},\frac3{4})\times (\frac1{4},\frac3{4})$ two open squares in $\mathbb R^2\simeq \C$. We fix a finite atlas $\cU$ of $X$ which satisfies the following: for any coordinate chart $U\in\cU$ with a diffeomorphism $\tau_U:U\to \tau_U(U)\subset \C$, the image $\tau_U(U)$ contains $\mathbb U$ and $\{\tau_U^{-1}(\mathbb U_0)\}_{U\in\cU}$ is an open cover of $X$. Then $\cU$ induces an atlas of $X\times X$ where the coordinate charts are given by $U\times U'$ along with a diffeomorphism $\tau_{U, U'}:=\tau_U\times\tau_{U'}$ from $U\times U'$ to an open subset in $\C^2$. They satisfy the following conditions:
\begin{enumerate}[label=(\roman*)]
    \item $\mathbb U^2\subset \tau_{U,U'}(U\times U')$; 
    \item $\{\tau_{U,U'}^{-1}(\mathbb U_0^2)\}_{U,U'\in\cU}$ is a cover of $X\times X.$ 
\end{enumerate}
 Therefore using a fixed partition of unity we may assume $\supp(\beta) \subset \tau_{U,U'}^{-1}(\mathbb U_0^2)$ for some $U$ and $U'$ as above. We will identify $U\times U'$ with $\tau_{U,U'}(U\times U')$ and use standard complex coordinates $(x,y)$ on $\C^2$. By linearity, it suffices to prove the theorems for the following three cases:
\begin{enumerate}
    \item $\beta = \phi(x,y) \d x\wedge \d \Bar{x}$;
    \item $\beta = \phi(x,y) \d y\wedge \d \Bar{y}$;
    \item $\beta = \phi(x,y) \d x\wedge \d \Bar{y} \quad \text{or} \quad \beta = \phi(x,y) \d y\wedge \d \Bar{x}$
\end{enumerate}
where $\phi$ is a $\mathcal C^5$ function on $\C^2$ supported by $\mathbb U_0^2$ with $\|\phi\|_{\mathcal C^5}\leq 1$.

\subsection*{Separating variables of \texorpdfstring{$\phi$}{phi}} 
Since $\supp(\phi)\subset \mathbb U_0^2$, we can extend $\phi$ to be a periodic function $\widetilde{\phi}$ on $\mathbb R^4$ of period 1 by defining
\[
    \widetilde{\phi}(x+z_1,y+z_2)=\phi(x,y) \quad \text{for any } x,y\in \mathbb U \text{ and } z_1,z_2\in\mathbb Z^2.
\]
Then $\widetilde{\phi}$ is $\mathcal C^5$ and $\|\widetilde{\phi}\|_{\mathcal C^k}=\|\phi\|_{\mathcal C^k}\leq 1$ for any $k\leq 5.$ Let $x=x_1+ \sqrt{-1}x_2$ and $y=y_1+ \sqrt{-1}y_2$. The Fourier expansion of $\widetilde{\phi}$ is given by 
\[
    \widetilde{\phi}(x_1,x_2,y_1,y_2)=\sum_I a_Ie^{2\pi  \sqrt{-1} I\cdot (x,y)}
\]
where $I=(i_1,i_2,i_3,i_4)\in \mathbb Z^4$. Recall that we define $|I|=\max_{1\leq s\leq 4}\{|i_s|\}$. Then as in (\ref{eq:decay-fouri-coeffi}) we have
\begin{equation}\label{eq:fouri-coeffi-bound}
    |a_I|\leq \frac{1}{|I|^5}
\end{equation}
for all $I\neq 0$. Moreover, $|a_I|\leq 1$ for all $I$. For a large integer number $N$ whose value will be specified later, we define the truncation function of $\widetilde{\phi}$ by
\[
    \widetilde{\phi}_N(x_1,x_2,y_1,y_2):=\sum_{|I|\leq N}a_Ie^{2\pi  \sqrt{-1} I\cdot (x,y)}.
\] 
Then we take $k=5$ in (\ref{eq:fouri-coeffi-bound}) and obtain that
\begin{equation}\label{eq:trunc-c0-norm}
    \|\widetilde{\phi}-\widetilde{\phi}_N\|_{\mathcal C^0}\leq \sum_{|I|>N}|a_I|\leq \sum_{|I|>N}\frac{1}{|I|^5}.
\end{equation}
To calculate the last sum, we notice that for any $m\in\mathbb N$, the number of all $I$ such that $|I|\leq m$ is $(2m+1)^4$. The number of all $I$ such that $|I|=m$ is thus given by $(2m+1)^4-(2m-1)^4\leq 80m^3$. Therefore,
\begin{equation}\label{eq:sum-I}
     \sum_{|I|>N}\frac{1}{|I|^5}\leq\sum _{m=N+1}^\infty\frac{1}{m^5}\cdot 80m^3\leq 80\sum_{m=N+1}^\infty \frac{1}{m^2}\leq \frac{80}{N}.
\end{equation}
Let $\chi:\C\to [0,1]$ be a smooth cut-off function on $\mathbb C$ which is supported on ${\mathbb U}$ and equals to 1 in a neighborhood of $\mathbb U_0$ and moreover $\|\chi\|_{\mathcal C^2}\leq 10.$ Denote by $\widetilde{\chi}(x,y)=\chi(x)\chi(y)$ which is a smooth function supported on ${\mathbb U}^2$ and equals to 1 on $\mathbb U_0^2$. Then $\phi=\phi\widetilde{\chi}=\widetilde{\phi}\widetilde{\chi}$.  

\begin{proof}[{End of the proof of \Cref{thm:main-thm-neq}}]

Recall that $d_1<d_2$ and $\Gamma_\infty=\pi^*_1(\mu)$. Let $\delta=d_1d_2^{-1}<1.$ 
\medskip

\textit{Case (1)}: Suppose $\beta=\phi(x,y) \d x\wedge \d \Bar{x}$. Then $\lp \Gamma_\infty, \beta\rp=\lp \mu, (\pi_1)_*\beta\rp=0.$ It remains to prove $\lp d_2^{-n}[\Gamma_n], \beta\rp$ goes to zero exponentially fast. Since $(\pi_1|_{\Gamma_n})_*\beta=\sum_{y\in f^n(x)} \phi(x,y)\d x\wedge \d \Bar{x}$, by definition of $d_1$ we have $\|(\pi_1|_{\Gamma_n})_*\beta\|_{\mathcal C^0}\leq d_1^n$. Therefore using (\ref{eq:graph-covering}), we obtain
\[
    |\lp d_2^{-n}[\Gamma_n], \beta\rp|=\left|d_2^{-n}\int_X(\pi_1|_{\Gamma_n})_*\beta\right|\leq A_1\delta^n
\]
where $A_1>0$ is a constant independent of $n$ and $\phi$.
\medskip

\textit{Case (2)}: Suppose $\beta=\phi(x,y) \d x\wedge \d \Bar{y}$. The proof is similar when $\beta=\phi(x,y) \d y\wedge \d \Bar{x}$. It is still true that $\lp \Gamma_\infty, \beta\rp=0.$ Recall that $\phi=\phi\widetilde{\chi}$, so we can write $\beta=\phi(x,y)\d x\wedge \widetilde{\chi}\d \Bar{y}$. By Cauchy-Schwartz inequality and Case (1), we have
\[
    |\lp d_2^{-n}[\Gamma_n],\beta\rp|^2\leq \lp d_2^{-n}[\Gamma_n], |\phi|^2 \sqrt{-1}\d x\wedge \d \Bar{x}\rp\lp d_2^{-n}[\Gamma_n], \widetilde{\chi}^2 \sqrt{-1} \d y\wedge \d \Bar{y}\rp\leq A_2\delta^n
\]
where $A_2>0$ is a constant independent of $n$ and $\phi$. Here in order to bound the third integral in the previous line, we also use the fact that $\|d_2^{-n}[\Gamma_n]\|\leq 2$ for all $n\geq 1$ which is proven right after (\ref{eq:graph-covering}).
\medskip

\textit{Case (3)}: Consider $\beta = \phi(x,y) \d y\wedge \d \Bar{y}$. Define $T_n=d_2^{-n}[\Gamma_n]-\Gamma_\infty$ as in (\ref{eq:T_n}) and recall that $\phi=\widetilde{\phi}\widetilde{\chi}$. For a fixed $n$ we can divide the term under consideration into two parts:
\begin{equation}\label{eq:trunc-T_n}
    \lp T_n,\phi(x,y)\d y\wedge \d \Bar{y} \rp=\lp T_n, (\widetilde{\phi}-\widetilde{\phi}_N)\widetilde{\chi}\d y\wedge \d \Bar{y} \rp+ \lp T_n, \widetilde{\phi}_N\widetilde{\chi}\d y\wedge \d \Bar{y} \rp.
\end{equation}
Since $\|T_n\|\leq 4$, the first term on the right hand side can be controlled by (\ref{eq:trunc-c0-norm}) and (\ref{eq:sum-I}):
\begin{equation*}
    |\lp T_n, (\widetilde{\phi}-\widetilde{\phi}_N)\widetilde{\chi}\d y\wedge \d \Bar{y} \rp|\leq 4\|\widetilde{\phi}-\widetilde{\phi}_N\|_{\mathcal C^0}\leq \frac{320}{N}.
\end{equation*}
For each $|I|\leq N$, we define 
\begin{align*}
    &\varphi_I(x)=e^{2\pi  \sqrt{-1} (i_1x_1+i_2x_2)}\chi(x),\\
    &\theta_I(y)=e^{2\pi \sqrt{-1}(i_3y_1+i_4y_2)}\chi(y) \d y\wedge \d\Bar{y}.
\end{align*}
Then $\varphi_I$ is a smooth function with $\|\varphi_I\|_{\mathcal C^2}\leq 120N^2$ and also note that $\|\theta_I\|_{\mathcal{C}^0}\leq 1$. Moreover,
\[
    \lp d_2^{-n}[\Gamma_n],\widetilde{\phi}_N\widetilde{\chi}\d y\wedge \d \Bar{y} \rp=\sum_{|I|\leq N}\lp d_2^{-n}[\Gamma_n], a_I\pi_1^*\varphi_I\wedge \pi_2^*\theta_I\rp.
\]
Using (\ref{eq:graph-covering}), we have
\[
    \lp [\Gamma_n], \pi_1^*\varphi_I\wedge \pi_2^*\theta_I\rp=\int_X(\pi_2|_{\Gamma_n})_*(\pi_1^*\varphi_I\wedge \pi_2^*\theta_I)=\lp (f^n)_*\varphi_I,\theta_I\rp.
\]
On the other hand, we also have
\begin{align}\label{eq:gamma-infty}
    \lp \Gamma_\infty, \pi_1^*\varphi_I\wedge \pi_2^*\theta_I\rp=\lp\mu,\varphi_I\rp\int_X\theta_I.
\end{align}
Note that the number of all $I\in\mathbb Z^4$ with $|I|\leq N$ is bounded by $90N^4$. Recall that $|a_I|\leq 1$ for all $I$. By \Cref{prop:equi-function-neq}, we obtain an estimate for the second term of (\ref{eq:trunc-T_n}):

\begin{equation}\label{eq:case3}
    \begin{aligned}
    |\lp T_n,\widetilde{\phi}_N\widetilde{\chi}\d y\wedge \d \Bar{y} \rp|
    &\leq\sum_{|I|\leq N}|\lp T_n, a_I\pi_1^*\varphi_I\wedge \pi_2^*\theta_I\rp|\\
    &=\sum_{|I|\leq N}\left|a_I\int_X\left[d_2^{-n}(f^n)_*\varphi_I-\lp\mu,\varphi_I\rp\right]\theta_I \right|\leq A_3N^6\delta^n
\end{aligned}
\end{equation}

where $A_3>0$ is independent of $n,N$ and $\phi$. Altogether we have
\[
    |\lp T_n,\phi(x,y)\d y\wedge \d \Bar{y} \rp|\leq A_4\left(N^6\delta^n+\frac{1}{N}\right)
\]
where $A_4=\max\{320,A_3\}$. This is true for all $N$. In order to get the final result, we choose $N=[\delta^{-n/12}]$ and deduce that $|\lp T_n,\phi(x,y)\d x\wedge \d \Bar{x} \rp|\leq C\delta^{n/12}$ for some $C>0.$
\end{proof}

\begin{remark}\rm
    We can extend \Cref{thm:main-thm-neq} to the case of meromorphic self-maps on a compact K\"ahler manifold $X$ of dimension $k$. We need to replace the condition $d_1<d_2$ by requiring the topological degree of $f$ (denoted by $d_t$) is strictly larger than other dynamical degrees, see \cite{DNT15} for definition of the dynamical degrees. There exists an equilibrium measure $\mu$ such that $f^*(\mu)=d_t\mu$. Let $\Gamma_n$ be the closure of the graph of $f^n$ in $X\times X$ and define $\Gamma_\infty:=\pi_1^*(\mu)$ where $\pi_1$ is the canonical projection from $X\times X$ to $X.$ Both $\Gamma_n$ and $\Gamma_\infty$ are positive closed $(k,k)$-currents on $X\times X$. It is proven in \cite[Lemma 5.2]{DNT15} that $d_t^{-n}[\Gamma_n]$ converges weakly to $\Gamma_\infty$. We have the following analogue of \Cref{thm:main-thm-neq}. The proof is a combination of our proof of \Cref{thm:main-thm-neq} and \cite[Lemma 5.2]{DNT15} therein.

    \begin{statement}{Theorem 1.1 (bis)}
        Let $f$ be a meromorphic self-map on a compact K\"ahler manifold $X$ of dimension $k$. Suppose the topological degree $d_t$ of $f$ is strictly larger than other dynamical degrees. Let $\mu,\Gamma_n$ and $\Gamma_\infty$ be as above. Then for every $\alpha>0$, there is a constant $0<\lambda_\alpha<1$ such that for any test $(k,k)$-form $\beta$ of class $\mathcal C^\alpha$ on $X\times X$, we have 
    \begin{equation*}
        \big| \langle d_t^{-n}[\Gamma_n]-\Gamma_\infty, \beta  \rangle \big|\leq C_\alpha\|\beta\|_{\mathcal C^\alpha}\lambda_\alpha^n, \quad \text{ for every }n\geq 1,
    \end{equation*}
     where $C_\alpha>0$ is a constant independent of $n$ and $\beta$.
    \end{statement}
\end{remark}

\begin{proof}[{End of the proof of \Cref{thm:main-thm-nwmod}}]
Now $f$ is non-weakly modular with degrees $d_1=d_2=d$ and $\Gamma_\infty=\pi_1^*(\mu^+)+\pi_2^*(\mu^-)$. 
\medskip

\textit{Cases (1) \& (2)}: In this setting, Cases (1) and (2) can be treated in the same way. For simplicity, we prove the case when $\beta=\phi(x,y)\d y\wedge \d \Bar{y}$. We follow the same lines as Case (3) in the proof of \Cref{thm:main-thm-neq} except some minor differences which we point out here. Different from (\ref{eq:gamma-infty}), in the present case we shall have
\[
    \lp \Gamma_\infty, \pi_1^*\varphi_I\wedge \pi_2^*\theta_I\rp=\lp\mu^+,\varphi_I\rp\int_X\theta_I.
\]
As before, let $T_n=d^{-n}[\Gamma_n]-\Gamma_\infty.$ The following is analogous to (\ref{eq:case3}) except that we use \Cref{prop:equi-function-nwmod} instead to get the last inequality:
\[
    |\lp T_n,\widetilde{\phi}_N\widetilde{\chi}\d y\wedge \d \Bar{y} \rp|
   \leq\sum_{|I|\leq N}\left|a_I\int_X\left[d^{-n}(f^n)_*\varphi_I-\lp\mu^+,\varphi_I\rp\right]\theta_I \right|\leq A_5N^6\lambda^n
\]
where $A_5>0$ is independent of $n,N$ and $\phi$. The rest is the same.
\medskip

\textit{Case (3)}: We prove the last case when $\beta=\phi(x,y)\d x\wedge \d\Bar{y}$ and note that the proof is analogous when $\beta=\phi(x,y)\d y\wedge \d\Bar{x}$. In this case, we always have 
\[
    \lp \Gamma_\infty, \beta\rp=\lp \mu^+, (\pi_1)_*\beta \rp+\lp \mu^-, (\pi_2)_*\beta \rp=0.
\]
Therefore it suffices to prove $\lp d^{-n}[\Gamma_n], \beta\rp$ goes to zero exponentially fast. Similar to the proof of \Cref{thm:main-thm-neq}, we divide this term into two parts:
\[
    \lp d^{-n}[\Gamma_n],\phi(x,y)\d x\wedge \d \Bar{y} \rp=\lp d^{-n}[\Gamma_n], (\widetilde{\phi}-\widetilde{\phi}_N)\widetilde{\chi}\d x\wedge \d \Bar{y} \rp+ \lp d^{-n}[\Gamma_n], \widetilde{\phi}_N\widetilde{\chi}\d x\wedge \d \Bar{y} \rp.
\]
Recall that $\|d^{-n}[\Gamma_n]\|\leq 2$. Again we use (\ref{eq:trunc-c0-norm}) and (\ref{eq:sum-I}) to have
\[
    |\lp d^{-n}[\Gamma_n], (\widetilde{\phi}-\widetilde{\phi}_N)\widetilde{\chi}\d x\wedge \d \Bar{y} \rp|\leq 2\|\widetilde{\phi}-\widetilde{\phi}_N\|_{\mathcal C^0}\leq \frac{160}{N}.
\]
To bound the second term, for each $I$ define 
\begin{align*}
    &\gamma_I(x)=e^{2\pi  \sqrt{-1} (i_1x_1+i_2x_2)}\chi(x)\d x;\\
    &\omega_I(y)=e^{2\pi  \sqrt{-1} (i_3y_1+i_4y_2)}\chi(y)\d \Bar{y}.
\end{align*}
Then $\|\gamma_I\|_{\mathcal C^0}\leq 1$ and $\|\omega_I\|_{\mathcal C^0}\leq 1$. By Cauchy-Schwartz inequality and \Cref{normless1}, we have 
\begin{align*}
    |\lp d^{-n}[\Gamma_n],\pi_1^*\gamma_I\wedge \pi_2^*\omega_I \rp|
    &=\left|\int_X d^{-n}(f^n)_*\gamma_I\wedge \omega_I\right|\\
    &\leq \|\omega_I\|_{L^2}\|d^{-n}(f^n)_*(\gamma_I)\|_{L^2}\leq A_6\lambda^n
\end{align*}
where $A_6$ is independent of $n,\gamma_I$ and $\omega_I$. As in (\ref{eq:case3}), we deduce that
\[
    |\lp d^{-n}[\Gamma_n], \widetilde{\phi}_N\widetilde{\chi}\d x\wedge \d \Bar{y} \rp|\leq \sum_{|I|\leq N}|a_I|\lambda^n\leq 90A_6N^4\lambda^n.
\]
Therefore,
\[
    |\lp d^{-n}[\Gamma_n],\phi(x,y)\d x\wedge \d \Bar{y} \rp|\leq A_7\left(N^4\lambda^n+\frac{1}{N}\right).
\]
where $A_7=\max\{160,90A_6\}$. Choose $N=[\lambda^{-n/8}]$ and deduce that $|\lp d^{-n}[\Gamma_n],\phi(x,y)\d x\wedge \d \Bar{y} \rp|\leq C\lambda^{n/8}$ for some $C>0.$ 
\end{proof}

{
  \preparebibliography
  \printbibliography
}

\end{document}